\documentclass[12pt,a4paper]{amsart}
\usepackage{amsfonts}
\usepackage{amsthm}
\usepackage{amsmath}
\usepackage{amssymb}
\usepackage{enumitem,kantlipsum}
\newtheorem{theorem}{Theorem}
\newtheorem{lemma}[theorem]{Lemma}
\newtheorem{rem}[theorem]{Remark}
\newtheorem{conj}[theorem]{Conjecture}
\newcommand{\M}{\mathfrak M}
\newcommand{\F}{\mathbb F}
\newcommand{\R}{\mathbb R}
\newcommand{\CC}{\mathbb C}
\newcommand{\A}{\mathcal A}

\usepackage[colorlinks=true, allcolors=black]{hyperref}

\theoremstyle{definition}

\theoremstyle{remark}

\numberwithin{equation}{section}
\DeclareMathOperator{\Image}{Im}
\DeclareMathOperator{\trace}{tr}
\DeclareMathOperator{\Sc}{Sc}
\DeclareMathOperator{\Char}{Char}
\DeclareMathOperator{\ssl}{sl}

\title[Images of non-associative multilinear polynomials]{The images of multilinear non-associative polynomials evaluated on a rock-paper-scissors algebra with unit over an arbitrary field.}
\author{Sergey Malev, Coby Pines}
\address{Ariel university, Ariel 40700, Israel\\ Department of Mathematics}
\email {sergeyma@ariel.ac.il}
\email {cobypinesdirac@gmail.com }
\thanks{We would like to thank A.Kanel-Belov and A. Domoshnitsky for interesting and fruitful discussions.}

\begin{document}

\begin{abstract}
Let $\F$ be an arbitrary field. We consider a  commutative, non-associative, $4$-dimensional algebra $\M$ of the rock, the paper and the scissors with unit  over $\F$ and we prove that the image over $\M$ of every non-associative multilinear polynomial over $\F$ is a vector space. The same question we consider for two subalgebras: an algebra of the rock, the paper and the scissors without unit, and an algebra of trace zero elements with zero scalar part. Moreover in this paper we consider the questions of possible evaluations of homogeneous polynomials on these algebras. 
\end{abstract}

\maketitle
\begin{flushright}
Dedicated to the 80-th anniversary of A.V. Mikhalev
\end{flushright}

\section{Introduction}
The study of images of polynomials evaluated on algebras is one of the most important branches of  modern algebra.
Similar questions for word maps in groups were considered in \cite{Bre,GKP}. Waring type problems for groups were investigated by Shalev \cite{Sha1,Sha2,Sha3}. Similar questions for matrix rings were investigated by Bre\v{s}ar (\cite{Bre}). A good survey describing these and other references can be found in \cite{BMRY}.

One of the central conjectures regarding possible evaluations of multilinear polynomials on matrix algebras was attributed to Kaplansky and formulated by L'vov in \cite{Dn}:
\begin{conj}[L'vov-Kaplansky] \label{kapl}Let $p$ be a non-commutative multilinear polynomial. Then the
set of values of $p$ on the matrix algebra $M_n(K)$ over an infinite field $K$ is a vector
space.
\end{conj}

It is well-known (\cite{BMR1,BMR2,BMR3,BMR4,M1,M2,BMRY}) that this conjecture can be reformulated as follows:
\begin{conj} If $p$ is a multilinear polynomial evaluated on the matrix ring
$M_n (K)$, then $\Image p$ is either $\{0\}, K, \ssl_n (K),$ or $M_n (K).$ Here $K$ indicates the set
of scalar matrices and $\ssl_n (K)$ is the set of  matrices with trace equal to zero.
\end{conj}

When $n=2$, for the case of $K$ being quadratically closed it was proved in \cite{BMR1}, and in \cite{M1} it was proved for the case of $K=\R$, and an interesting result was obtained for arbitrary fields.

For $n>2$ this question was considered in \cite{BMR2,BMR3,BMRY} and partial results were obtained.
In \cite{M2} the same question was considered for the algebra of quaternions with the Hamilton multiplication and it was shown that any evaluation of a multilinear polynomial  is a vector space. In the same paper it was said that this question is interesting only for simple algebras, since for non-simple algebras it may be answered negatively. Indeed, this conjecture fails for the Grassmann algebra.

Nevertheless, there is an interest in the investigation of this question for non-simple algebras: for some of them the Kaplansky question can be answered positively. For example, if we consider the $8$-dimensional algebra $M_2(K) \oplus M_2(K)$ (which is not simple) then it is easy to see that the evaluation of any multilinear polynomial is a pair of its evaluations on $M_2(K)$ and thus, the only possible multilinear evaluations are vector spaces:
$$ \{0\}, K\oplus K, \ssl_2 (K)\oplus\ssl_2(K), \ \text{or}\  M_2 (K)\oplus M_2(K).$$

In this paper we consider a non-simple algebra, defined in Section \ref{prel}. Unlike the previous results regarding this area, this algebra is non-associative and commutative and, although  we consider non-associative commutative evaluations, the answer to the L'vov-Kaplansky conjecture for this algebra is positive.

Previously, non-associative algebras (in particular the algebra of Cayley numbers) were considered in \cite{ZSSS}. The question of possible
multilinear non-associative evaluations was considered in \cite{BMR4} where the Kaplansky conjecture was considered for Lie polynomials.

In \cite{BMR1,BMR2,BMR3,BMR4,M1,M2,BMRY} the question of possible semi-homogeneous evaluations was investigated, and here we consider such a question as well. Unfortunately, we have not succeeded to answer it. However, in Section \ref{homogen} we formulate interesting conjectures and discuss them.

\section{Preliminaries}\label{prel}

Let $(X,\cdot)$ be a finite monad i.e $X$ is a finite set together with a binary operation $\cdot$ for which there is a unit element with respect to~$\cdot$.
Let $\F$ be an arbitrary field and denote by $M_{ \F}(X)$ the free vector space over  $\F$ with basis set $X$. By extending multiplication from  $X$ to $M_{ \F}(X)$ bilinearly, we obtain a non-associative algebra with unit. We call $M_{ \F}(X)$ the monad algebra of $X$ over $\F$.\newline   

We now consider the monad $X=\{1,R,P,S\}$ with multiplication defined as follows: multiplication is commutative, $1$ serves as the unit element, all elements of $X$ are idempotent and  $R\cdot P =P$, $P\cdot S =S$ and $S\cdot R =R$. This monad is the well known rock- paper- scissors magma, where each element is idempotent and the product is the winner in one of the most well-known games.\newline
In what follows, we shall write $\M:=M_{\F}(X)$. It follows that $\M$ is a commutative, non-associative, $4$-dimensional algebra with unit over $\F$.\newline
When evaluating polynomials on non-associative algebras, we must consider non-associative polynomials. A non-associative polynomial over a field $\F$ is, intuitively, a polynomial over $\F$ in which the placing of brackets in each monomial matters.
For instance, the polynomials: 
$$
p_1(x_1,x_2,x_3,x_4)=x_1(x_2(x_3x_4)),\ 
p_2(x_1,x_2,x_3,x_4)=(x_1x_2)(x_3x_4)
$$
are considered to be  different non-associative polynomials.\newline

Let $\A$ be a non-associative algebra over a field $\F$ and let $p(x_1,x_2,\dots,x_m)$ be a non-associative polynomial over $\F$.\newline
The associated polynomial function is denoted by $\Tilde{p}$. \newline
The image of $p$ over $\A$, $\Image _{\A}(p)$, is the evaluation of the associated polynomial function $\Tilde{p}$ on the  algebra $\A$. Usually the choice of the algebra is evident, and we write simply $\Image p$.

A non-associative polynomial over a field $\mathbb{F}$ is called  multilinear
if it is linear with respect to each of the variables. This means in  particular, that each variable appears exactly once in each monomial.

For $x=x_01+aP+bR+cS \in\M$, we define
the scalar part of $x$ to be $\Sc(x):=x_0$ and
the trace of $x$ to be $\trace(x):=x_0+a+b+c$.\newline
It is obvious that $ \Sc\colon {\M}\rightarrow \F$ and $ \trace \colon {\M}\rightarrow \F$ are linear functions on the vector space ${\M}$.
Moreover, they also respect multiplication on ${\M}$ i.e for $x,y \in {\M}$:
\begin{equation*}
 \Sc(x \cdot y)=\Sc(x)\Sc(y)
 \end{equation*}
\begin{equation*}
\trace(x \cdot y)=\trace(x)\trace(y)
 \end{equation*}
Therefore, both functions $\Sc(*)$ and $\trace(*)$ are homomorphisms from $\M$ to $\F$.

We denote by $\M_{0}$ the set of elements of $\M$ with zero trace and zero scalar part.\newline
It follows that $\M_{0}$ is a two dimensional subspace of $\M$, and moreover, it is an ideal and a subalgebra of $\M$.

\section{Main Theorem}
\begin{theorem}\label{main}
Let $\F$ be an arbitrary field. If $p(x_1,x_2,\dots,x_m)$ is a  multilinear commutative non-associative  polynomial with coefficients in $\F$, then the evaluation of $p$ over $\M$ is either:
\begin{enumerate}
\item $\{0\}$;
\item $\left<P+R\omega+S\omega^2\right>$;
\item $\left<P+R\omega^2+S\omega\right>$;
\item ${\M}_{0}$; 
\item $\M$.
\end{enumerate}
\end{theorem}
By $\omega$ we denote $1$ if $\Char\F=3$, and a primitive cube root of $1$ otherwise (if such an element exists in $\F$). If such element $\omega$ does not exist in $\F$, options (2) and (3) are impossible. As well-known, this element ($\omega$) satisfies $\omega^2+\omega+1=0$. In particular, such an element exists in $\CC$, although it does not exist in  $\R$.
We will use the following basic fact from linear algebra: 
\begin{lemma}[\cite{M1}, Lemma~3]\label{2 dim lemma}
Let $L$  be a vector space over a field $\F$ and suppose that  $f\colon L\times \cdots \times L\rightarrow L$  is a multilnear map. Assume that  $\Image(f)$ contains two  vectors which are not proportional. Then $\Image(f)$ contains a two dimensional subspace. In particular, if $\Image(f)$  is contained in a  two dimensional subspace $M$, then $\Image(f)=M$.
\end{lemma}
\begin{rem}\label{lowdim}
As a consequence, one can immediately conclude from Lemma \ref{2 dim lemma}, that the image set of any multilinear map (in particular of a multilinear polynomial) is either a vector space or at least a $3$-dimensional set. By the dimension of an image set, we  mean the dimension of its Zariski closure. Note that a polynomial image is not necessarily Zariski closed.
\end{rem}
Let us first prove the following Lemma:
\begin{lemma}\label{lemma-vec}
Let $\F$ be an arbitrary field. If $p(x_1,x_2,\dots,x_m)$ is a  multilinear commutative non-associative  polynomial with coefficients in $\F$, then the evaluation of $p$ over $\M$ is either:
\begin{enumerate}
\item $\{0\}$;
\item some one-dimensional vector space;
\item ${\M}_{0}$; 
\item $\M$.
\end{enumerate}
\end{lemma}

\begin{proof}
Since  $p(x_1,\dots,x_m)$ is multilinear, it is a linear combination of monomials such that, in each monomial, every variable $x_i$ appears exactly once.
Let the coefficients of the monomials of $p$ be  $c_1,\dots,c_d$. Suppose that $c:=\sum\limits_{i=1}^{d}c_i\neq 0$. let $a\in \M$. Then $a=\Tilde{p}(c^{-1}\cdot a,1,\dots,1)$ and thus $\Image(p)=\M$.\newline
Thus, suppose that $c=0$. If $a\in \Image(p)$, then there exist elements  $a_1,\dots,a_m \in \M$ such that $a=\Tilde{p}(a_1,\dots,a_m)$. Since $\F$ is commutative and associative with respect to multiplication  and $ \trace$ and $\Sc$ are linear and respect multiplication, it follows that \begin{equation*}
  \trace(a)=\trace(a_1)\cdots \trace(a_m) \cdot \sum_{i=1}^{d}c_i=0  
\end{equation*} 
and similarly $\Sc(a)=0$.
Hence, $\Image(p)$ is contained in $\M_{0}$. Note that $\dim\ \M_0=2$. Thus, according to Remark \ref{lowdim}, $\Image p$ is a vector space: either $2$-dimensional (and thus coincides with $\M_{0}$), or of dimension $1$ or $0$. In the last case $p$ is a PI.
\end{proof}

We now give the proof of Theorem \ref{main}:
\begin{proof}
In view of Lemma \ref{lemma-vec}, we have to classify $1$-dimensional images. Consider the linear map $\varphi:\M\rightarrow\M$ defined as follows:
$$1\mapsto 1,\ P\mapsto R, R\mapsto S, S\mapsto P.$$
It is not difficult to see that this map respects multiplication and thus is an automorphism of $\M$.
Therefore, if $p(x_1,\dots,x_m)$ is a multilinear polynomial evaluated on $\M$, then for any values of $x_i$ we have  $$\varphi(p(x_1,\dots,x_m))=
p(\varphi(x_1),\dots,\varphi(x_m)).$$
Hence, if $\Image p$ is a $1$-dimensional vector space spanned by an element $x=aP+bR+cS\in\M_0$, $\varphi(x)$ should be proportional to $x$ i.e it should span the vector space. Note that $\varphi(aP+bR+cS)=cP+aR+bS$. Thus, $\frac{c}{a}=\frac{a}{b}=\frac{b}{c}$ and it is not difficult to see that this ratio can be only a cube root of $1$. If $\Char\F=3$, the only element in $\F$ satisfying this property is $1$ and thus, $\Image p=\left<P+R+S\right>$. If $\Char\F\neq 3$ and $\omega\in\F$, there are three elements satisfying this property: $1,\omega$ and $\omega^2$. However, the option $a=b=c$ is impossible since $\trace(P+R+S)=3$ and this is not an element of $\M_0$. 
\end{proof}
\section{Examples}

The most important thing to understand is whether or not polynomials with such images exist. Of course, it is very simple to construct an example of a polynomial whose image set is $\M$: one can take $p(x)=x$, or any other multilinear polynomial with nonzero sum of coefficients. 

It is not difficult to see that the set $\M_0$ can also be achieved, for instance, as the image of the polynomial $g(x,y,z)=(xy)z-(xz)y.$ Indeed, its image contains the element 
$g(P,R,S)=(PR)S-P(RS)=PS-PR=S-P$ and thus, by Theorem \ref{main}, $\Image g=\M_0$.
Unfortunately, we have not succeeded in constructing examples of multilinear polynomials with $1$-dimensional images. However, most likely they exist. Polynomial identities exist as well. Indeed, the computation shows that the polynomial $f(x,y,z,t)=(xy)(zt)-(xz)(yt)$ evaluated on $\M_0$ is a PI (in Lemma~\ref{pi0}, we prove this for the case $\Char\F\neq 3$. However, this is true for arbitrary fields and can be easily checked by basis element evaluations. For example, one can take  the elements $P-R$ and $R-S$ as a basis for $\M_0$ and check all $16$ evaluations). We can take four polynomials with image sets $\M_0$ and put them instead of $x,y,z$ and $t$ inside $f$: 
$x=g(x_1,x_2,x_3),\ y=g(y_1,y_2,y_3),\ z=g(z_1,z_2,z_3)$ and $t=g(t_1,t_2,t_3)$. As a result, we obtain a multilinear polynomial in $12$ variables which is a commutative non-associative polynomial identity of the algebra $\M$. Of course, $12$ is not the minimal possible degree for polynomial identities of $\M$.

\section{PI algebras}
Moreover, any finite dimensional commutative non-associative/non-commutative associative/non-commutative non-associative algebra is a PI algebra with nontrivial multilinear PI (i.e. PI of the same type as a type of an algebra). 
Indeed, let us compute the number of possible multilinear monomials of degree $m$: the number of associative non-commutative monomials is $m!$, the number of non-associative non-commutative monomials is $m!\cdot C_{m-1},$ where $C_{m-1}=\frac{1}{m} \binom{2m-2}{m-1}$ is a Catalan number.
In our case, we are interested in multilinear non-associative commutative monomials. Each monomial has exactly $m-1$ multiplications and for each of them we can change places of the multipliers. Thus, there are exactly $2^{m-1}$ different non-commutative non-associative monomials, corresponding to each commutative non-associative monomial, and therefore, the number of commutative non-associative monomials is  $m!\cdot C_{m-1}\cdot 2^{1-m}.$
If our algebra is finite dimensional, let $d$ be its dimension, $A=\left<E_1,\dots, E_d\right>$. In this case, a nonzero multilinear polynomial $p(x_1,\dots,x_m)$ is a PI if and only if its evaluations on all sets of basis elements $E_i$ is zero. We have exactly $d^m$ such evaluations, each of them has $d$ coordinates, and therefore, we have a system of $d^{m+1}$ linear equations, where the unknowns are the coefficients of $p$. The number of unknowns is the number of possible monomials. Hence, the existence of a PI follows from existence of a number $m$, such that the number of monomials  is larger than $d^{m+1}$. Remember, that for each $m$, these numbers (depending on the type of algebra) are $m!,\ m!\cdot C_{m-1}$, and $m!\cdot C_{m-1}2^{m-1}.$ For large $m$, these numbers exceed $d^{m+1}$ and thus, such  an $m$ exists.

\section{Subalgebras, good basis, automorphisms and polynomial identities}
Let us consider the algebra $\M_0$ separately. Unlike $\M$,  this is a simple algebra, i.e it does not contain any non-trivial ideals. In this section we assume that $\Char\F\neq 3$ and $\omega\in \F$. The second condition can be achieved if instead of the field $\F$, we consider either its algebraic closure or its extension $\F[\omega]$.
In this case, $\M_0$ has the following basis (called "the good basis"):
$$U=\frac{1+2\omega}{3}\left(P+R\omega+S\omega^2\right),$$
$$V=\frac{1+2\omega^2}{3}\left(P+R\omega^2+S\omega \right).$$
Note that the condition $\Char\F\neq 3$ is important:  if $\Char\F=3$ these elements $U$ and $V$ are not well defined.
A simple computation shows that $$U^2=V,\ V^2=U, \ \text{and}\ UV(=VU)=0.$$
The automorphism $\varphi$ of order $3$ defined in the proof of Theorem \ref{main}, induces an automorphism of $\M_0$:
$\varphi(U)=\omega^2 U,\ \varphi(V)=\omega V.$
\begin{rem}
There is another important automorphism $\psi$ of $\M_0$: $U\mapsto V,\ V\mapsto U$. This automorphism cannot be extended to $\M$. Nevertheless, if $\F=K[\omega]$ for some subfield $K$ which does not contain $\omega$, this automorphism can be extended from $\M_0(K)$ to $\M(\F)$ considered as an algebra over $K$: $U\mapsto V,\ V\mapsto U, \omega\mapsto\omega^2,\omega^2\mapsto\omega.$ The second two evaluations define the conjugation automorphism of $\F$ preserving $K$. In the usual basis $\{1,P,R,S\}$ of $\M$ it leaves basic elements and conjugates field coefficients only. The order of $\psi$ is $2$.

Note that not every field $\F$ containing $\omega$ can be presented as $K[\omega]$. For instance, the field $\F_7$ has an element $\omega=2\in\F_7$ but does not have any subfields at all.

\end{rem}

Now we are ready to prove a lemma about polynomial identities of $\M_0$:
\begin{lemma}\label{pi0}
The polynomial $f(x,y,z,t)=(xy)(zt)-(xz)(yt)$ is a PI of $\M_0$.
\end{lemma}
\begin{proof}
Consider the evaluations of $f(x,y,z,t)$ on the basis elements $U$ and $V$: for such evaluations, a product of two different elements equals zero, and a product of the type $(xy)(zt)$ is not zero when $x=y$, $z=t$ and $xy=zt$, which happens if and only if $x=y=z=t$. However, in this case $(xy)(zt)=(xz)(yt)$ and therefore $f$ is PI.
\end{proof}
\begin{rem}
Note that if a multilinear polynomial $p(x_1,\dots,x_m)$ has a $1$-dimensional image, i.e. either $<U>$ or $<V>$,
it is PI of $\M_0$.
Indeed, its evaluation on $\M_0$ must either coincide with its evaluation on $\M$ or it must be $\{0\}$. However, it should be invariant under the automorphism $\psi$, which is possible only in the case that $p$ is PI of $\M_0$. Nevertheless, this  does not imply that such a polynomial does not exist: even in the case  $\F=K[\omega]$, where  $\psi$ can be extended to $\M$, we need to conjugate elements of $\F$, and this will change the coefficients of the polynomial $p$. The problem considering the  existence  of such polynomials remains being open.
\end{rem}
The possible multilinear evaluations on $\M_0$ are described in
\begin{theorem}
Let $\F$ satisfy $\Char\F\neq 3$. If $p(x_1,x_2,\dots,x_m)$ is a  multilinear commutative non-associative  polynomial with coefficients in $\F$, then the evaluation of $p$ over ${\M}_0$ is either:
\begin{enumerate}
\item $\{0\}$;
\item ${\M}_{0}$; 
\end{enumerate}
\end{theorem}
\begin{proof}
Indeed, there are no $1$-dimensional multilinear evaluations since, if $\Image p=<aU+bV>$, $aU+bV$ must be proportional to $\varphi(aU+bV)=a\omega^2 U+b\omega V$ which is possible only if one of the coefficients $a$ or $b$ equals zero. However, this is also impossible, since in this case, $\Image p$ is not invariant under $\psi$.
\end{proof}
Another important subalgebra of $\M$ is the  subalgebra $\tilde{\M}$: the rock-paper-scissors without unit, i.e. the kernel of the homomorphism $\Sc$. Here we take the three element basis $U,V$ (which were already defined) and $W=\frac{1}{3}(P+R+S)$. 
In this case, 
$$W^2=W,\ WU=\frac{1+2\omega}{3}U,\ \text{and}\  WV=\frac{1+2\omega^2}{3}V.$$
The problem as to whether a multilinear evaluation  on $\tilde\M$ must be a vector space, remains being open. However, we can prove the following:
\begin{lemma}\label{tilde}
Let $\F$ satisfy $\Char\F\neq 3$. If $p(x_1,x_2,\dots,x_m)$ is a  multilinear commutative non-associative  polynomial with coefficients in $\F$, then the evaluation of $p$ over $\tilde\M$ is either:
\begin{enumerate}
\item $\{0\}$;
\item $\left<U\right>$;
\item $\left<V\right>$;
\item $\left<W\right>$;
\item ${\M}_{0}=\left<U,V\right>$; 
\item $\left<U,W\right>$;
\item $\left<V,W\right>$;
\item Zariski dense in $\tilde\M$.
\end{enumerate}
\end{lemma}

\begin{proof}
According to Lemma \ref{2 dim lemma}, $\Image p$ must be either a vector space or at least $3$-dimensional.  

If the  dimension  of $\Image p$ is no more than $2$, it is a vector space. If $\dim\Image p= 0$, then $p$ is PI and we have the case (1). If $\dim\Image p=1$, $\Image p$ is a line which is invariant under the automorphism $\varphi$, and we have one of the cases (2)-(4). If $\dim\Image p= 2$, we have a two-dimensional subspace invariant under $\varphi$.
Consider some evaluation of $p$: $q=aU+bV+cW$ with at least two nonzero coefficients
$a,b,c\in\F$.
Note that $\varphi(U)=\omega^2 U, \varphi(V)=\omega V, \varphi(W)=W.$ Thus, the elements $\varphi(q)=a\omega^2 U+b\omega V+cW$ and $\varphi^2(q)=a\omega U+ b\omega^2 V+cW$ belong to the image of $p$. The linear span of these three elements is the linear span of the basis elements $\{U,V,W\}$.
 In particular, we obtain that one of the coefficients of $q$ must zero, and the ($2$-dimensional) image of $p$ contains a plane spanned by two basis elements (and thus coincides with it).
Therefore, we have one of the cases (5)-(7).
Finally, if $\dim\Image p= 3$,  the image is a Zariski dense subset, and we have the case (8).
\end{proof}

Note that not all these options are possible:
\begin{theorem}
Let $\F$ satisfy $\Char\F\neq 3$. If $p(x_1,x_2,\dots,x_m)$ is a  multilinear commutative non-associative  polynomial with coefficients in $\F$, then the evaluation of $p$ over $\tilde\M$ is either:
\begin{enumerate}
\item $\{0\}$;
\item $\left<U\right>$;
\item $\left<V\right>$;
\item ${\M}_{0}=\left<U,V\right>$; 
\item Zariski dense in $\tilde\M$.
\end{enumerate}
\end{theorem}
\begin{proof}
Options $(4),(6)$ and $(7)$ in Lemma \ref{tilde} are impossible since if the sum of coefficients of $p$ is $c\neq 0$, the element $p(I,I,I,\dots,I)=cI$ belongs to the image of $p$, for $I$ being arbitrary idempotent of $\tilde M$. In particular, $P,R$ and $S$ are idempotents. In this case the linear span of $\Image p$ must be equal to $\tilde M$, which does not hold for these three cases. 
\end{proof}

Unfortunately, the question of possibility of cases $(2)$ and $(3)$  remains being open. The other problem is, whether or not the polynomial in the last case  must be surjective.

\section{Semi-homogeneous polynomials}\label{homogen}
In this section, we have more questions than answers. Nevertheless, we  can definitely  claim the following: the evaluation of any such polynomial should be invariant with respect to the automorphism $\varphi$.
We have a conjecture:
\begin{conj}\label{conj-1var}
Any homogeneous polynomial in one variable with nonzero sum of coefficients, has a Zariski dense image set.
\end{conj}
This is an interesting question to study. Indeed, considering monomials depending on one variable, there are different monomials of the same degree. For instance, the monomial $(x^2)^2$ is not the same as $x(x(x^2))$. Not only as monomials are they different: they also have different evaluations.
For example, if $x=P+R-2S$ then  $(x^2)^2=9(R-P)$ and
 $x(x(x^2))=9(P-R)$. Of course, if $\Char\F=2$, this is the same. The case $\Char \F=2$ can be considered separately. In  particular, for $\F=\F_2$ being the field of two elements, the evaluations of $(x^2)^2$ and $x(x(x^2))$ coincide for every $x$. Moreover, evaluations of any two monomials in one variable of equal degree coincide, which makes the monomial function $x^d$ being well defined.

If Conjecture \ref{conj-1var} holds, we can conclude the following:
\begin{conj}
Let $p(x_1,\dots, x_m)$ be any semi-homogeneous polynomial evaluated on $\M$ of degree $d$, and suppose that the field $\F$ is closed under  $d$-roots. Then the image set of $p$ satisfies one of the following conditions:
\begin{enumerate}
\item $\Image p=\{0\}$;
\item $\Image p$=$\left<P+R\omega+S\omega^2\right>$;
\item $\Image p$=$\left<P+R\omega^2+S\omega\right>$;
\item $\Image p$ is Zariski dense in ${\M}_{0}$; 
\item $\Image p$ is Zariski dense in $\M$.
\end{enumerate}
\end{conj}
This conjecture follows from the previous one:
if the sum of the coefficients is not zero, we can consider the polynomial $q(x)=p(x,x,\dots,x)$, which according to Conjecture \ref{conj-1var}, has a Zariski dense image set  in $\M$, and its image is a subset of the image of $p$.
If the sum of the coefficients is zero, then, as we know,both $\Sc(p)$ and $\trace(p)$ vanish on each evaluation and hence $\Image p\subseteq\M_0$. The image of any semi-homogeneous polynomial is a cone, and  therefore, there are three options:The image is $2$-dimensional and its Zariski closure is $\M_0$; The image  is $1$-dimensional and its image is one of the two possible lines; The image is $0$-dimensional, and in this case, $p$ should be a PI.
Note that if conjecture \ref{conj-1var} does not hold, then there is one more possible option for the $1$-dimensional image: the line $\left<P+R+S\right>$.

\section{Conclusion}
Questions related to the evaluations of multilinear and homogeneous polynomials are actual and applicable. This work continues series of works \cite{BMR1,BMR2,BMR3,BMR4,M1,M2,BMRY}.


\begin{thebibliography}{9}

\bibitem[BeMR1]{BMR1}   Belov, A.; Malev, S.; Rowen, L., {\em The images of noncommutative polynomials evaluated on $2\times 2$ matrices,} Proc. Amer. Math. Soc.
{\bf 140} (2012), 465--478.

\bibitem[BeMR2]{BMR2} Belov, A.; Malev, S.; Rowen, L., {\em The images of multilinear  polynomials evaluated on $3\times 3$ matrices}, Proc. Amer. Math. Soc. 
{\bf 144} (2016), 7--19.

\bibitem[BeMR3]{BMR3} Belov, A.; Malev, S.; Rowen, L., {\em Power-central polynomials on matrices,}  Journal of Pure and Applied Algebra {\bf 220} (2016), 2164--2176.


\bibitem[BeMR4]{BMR4} Belov, A.; Malev, S.; Rowen, L., {\em The images of Lie polynomials evaluated on matrices,} Communications in Algebra {\bf 45} (11) (2017), 4801--4808

\bibitem[BeMRY]{BMRY} Belov, A.; Malev, S.; Rowen, L., Yavich, R.{\em Evaluations of noncommutative polynomials on algebras: Methods and problems, and the L'vov-Kaplansky Conjecture.,} preprint, available on arXiv:1909.07785 (2020).

\bibitem[Bre]{Bre} Bre\v{s}ar, M., {\em Commutators and images of noncommutative polynomials}, preprint
available on arXiv:2001.10392 (2020).

\bibitem[Dn]{Dn}
{\em DNIESTER NOTEBOOK: Unsolved Problems in the Theory of Rings
and Modules}, Mathematics Institute, Russian Academy of Sciences
Siberian Branch, Novosibirsk Fourth Edition, (1993)

\bibitem[GKP]{GKP}
Gordeev, N., Kunyavskii, B., Plotkin, E. {\em Geometry of word equations in simple algebraic groups over special fields (Russian. Russian summary)},
Uspekhi Mat. Nauk \textbf{73} (2018), no. 5(443), 3--52; translation in Russian Math. Surveys
\textbf{73} (2018), no. 5, 753--796.

\bibitem[Mal1]{M1} Malev, S., {\em The images of noncommutative polynomials evaluated on $2\times 2$ matrices over an arbitrary field,}
Journal of Algebra and its Applications \textbf{13}(2014), no.6., 145004, 12 pp.

\bibitem[Mal2]{M2} Malev, S., {\em The images of noncommutative polynomials evaluated on the Quaternion algebra,}
Journal of Algebra and its Applications (2021), 8 pp. https://doi.org/10.1142/S0219498821500742
\bibitem[Sh1]{Sha1}
Shalev, A.  {\em Commutators, words, conjugacy classes and character methods},  Turkish J. Math. {\bf 31} (2007), 131--148.

\bibitem[Sh2]{Sha2}
Shalev, A.  {\em Word maps, conjugacy classes, and a noncommutative
Waring-type theorem}, Annals of Math., \textbf{170} (2009),
1383--1416.

\bibitem[Sh3]{Sha3}
Shalev, A.  {\em Some results and problems in the theory of word maps}, in: ``Erd\"os Centennial'' (L. Lov\'asz, I. Ruzsa, V. T. S\'os, D. Palvolgyi, Eds.), Bolyai Soc. Math. Studies, {\bf 25}, Springer, (2013), 611--649.
\bibitem[ZhSlSheShi]{ZSSS}
Zhevlakov, K.; Slinko, A.; Shestakov, I.; Shirshov, A. {\em Rings close to associative}, Moscow, Nauka (1978), 432 pp.
\end{thebibliography}
\end{document}